\documentclass[12pt]{article}
\usepackage{geometry}                
\usepackage{graphicx}
\usepackage{amssymb, mathrsfs, enumerate}
\usepackage{bbm}

\usepackage{amscd}
\usepackage{amssymb}
\usepackage{graphics}
\usepackage{amsmath}
\usepackage[english]{babel}
\usepackage{hyperref}
\usepackage[latin1]{inputenc}

\usepackage{color}      
\usepackage{srcltx}  

\usepackage{amsthm, amssymb}
\usepackage{setspace}
\usepackage[all]{xy}
\usepackage{graphicx}
\usepackage{ifpdf}
\usepackage{babel}
\usepackage{verbatim}

\newtheorem{prop}{Proposition}
\newtheorem{proposition}[prop]{Proposition}

\newtheorem{theorem}[prop]{Theorem}
\newtheorem{thm}[prop]{Theorem}
\newtheorem{cor}[prop]{Corollary}
\newtheorem{lemma}[prop]{Lemma}

\newtheorem{remark}[prop]{Remark}

\numberwithin{prop}{section}

\newtheorem{conj}[prop]{Conjecture}

\DeclareMathOperator{\Gal}{Gal}

\DeclareMathOperator{\h}{ht}
\DeclareMathOperator{\GL}{GL}
\DeclareMathOperator{\supp}{supp}
\DeclareMathOperator{\Hom}{Hom}
\DeclareMathOperator{\Spec}{Spec}
\DeclareMathOperator{\Cay}{Cay}

\DeclareMathOperator{\Frac}{Frac}

\DeclareMathOperator{\Mat}{Mat}

\DeclareMathOperator{\End}{End}

\newcommand{\BC}{{\mathbb{C}}}
\newcommand{\BF}{{\mathbb{F}}}

\newcommand{\gC}{\Gamma}
\newcommand{\gc}{\gamma}

\newcommand{\ga}{\alpha}

\newcommand{\Ad}{\text{Ad}}

\newcommand{\SL}{\text{SL}}

\newcommand{\SO}{\text{SO}}
\newcommand{\Lie}{\text{Lie}}

\title{Mixed identities in linear groups --- effective version}

\author{Nir Avni and Tsachik Gelander}

\begin{document}

\maketitle

\begin{center}
{\it To Alex Furman for his 60th birthday}
\end{center}

\begin{abstract}
Let $\gC$ be a finitely generated linear group that does not satisfy any nontrivial (mixed) identity. We prove that there is a constant $C$ such that \begin{enumerate}
\item For every nontrivial mixed word $w(x)$ there is an element $\gamma$ of length $\| \gamma \|\leq C\log (2+\|w\|)$ such that $w(\gamma)\neq1$.
\item For every $n$ there is an element $\gamma\in \Gamma$ of length $\| \gamma \| \leq Cn$ such that $w(\gamma)\neq1$ simultaneously for all nontrivial mixed words of length $\|w\| \leq n$.
\end{enumerate} 
The proof uses a new probabilistic version of the super approximation theorem which holds for groups whose entries are not assumed to be algebraic.
\end{abstract}

{
  \hypersetup{linkcolor=blue}
}

\section{Introduction}
\subsection{Mixed identities}
Let $\gC$ be a group. A {\it mixed word} in $\Gamma$ is an expression of the form
$$
 w(x)=g_1x^{n_1}g_2x^{n_2}\cdots g_kx^{n_k},
$$
where $g_i\in \gC$ and $x$ is a free variable. We think of $w(x)$ as an element of the free product of $\Gamma$ and a cyclic group $\langle x \rangle$. We say that $\gC$ satisfies the identity $w(x)=1$ if $w(\gc)=1$ for all $\gc\in \gC$. For example, if $\alpha$ is a nontrivial central element in $\gC$ then $\gC$ satisfies the mixed identity $\alpha x \alpha^{-1}x^{-1}=1$. We say that $\gC$ is {\it mixed identity free}, or MIF, if it does not satisfy any nontrivial mixed identity, i.e., for every $w(x)\in \gC*\langle x \rangle \smallsetminus\{1\}$ there is $\gc\in \gC$ such that $w(\gc)\ne 1$. If $\Gamma$ is MIF then $\Gamma$ does not satisfy nontrivial mixed identities with several variables. Using commutators, we will show in Corollary \ref{Cor:simultaneously} that if $\gC$ is MIF then, for every finite set of words $\{w_1(x),\ldots w_m(x)\}\subset \gC*\langle x \rangle \smallsetminus\{1\}$, there is an element $\gc\in \gC$ that violates all $w_i$ simultaneously. 

We will show below that if $\Gamma$ is linear and MIF then the amenable radical of $\Gamma$ is trivial, so $\Gamma$ has a faithful representation $\rho$ for which the connected component of the Zariski closure of $\rho(\Gamma)$ is semisimple. Conversely, when $\Gamma$ is Zariski dense in a classical group, Tomanov \cite{Tomanov1,Tomanov2} gives necessary and sufficient conditions for MIF. For example:
\begin{itemize}
\item If $\overline{\Gamma}^Z \cong\SL_d$ then $\gC$ is MIF iff it has a trivial center.
\item If $\overline{\Gamma}^Z \cong\text{Sp}_{2d}$ and $\gC$ has no elements of order $2$, then $\gC$ is MIF.
\item If $\overline{\Gamma}^Z \cong\SO_d$ and $\gC$ has no nontrivial element $\gc$ for which $\gc+\gc^{-1}$ is a scalar matrix, then $\gC$ has MIF.
\end{itemize}

\subsection{Effective estimates}

Suppose that $\gC$ is generated by a finite set $X$. For an element $\gamma \in \Gamma$ we denote the word length of $\gamma$ with respect to the generating set $X$ by $\| \gamma \|_X$. Similarly, for $w(x)\in \Gamma * \langle x \rangle$, we denote the word length of $w$ with respect to the generating set $X\cup \left\{ x , x ^{-1} \right\}$ by $\|w\|_{X\cup \left\{ x, x ^{-1} \right\}}$. For a nontrivial mixed word $w(x)$ and a natural number $n$, define
\[
f_X(w):=\min \left\{ \| \gamma \|_X \mid w(\gamma)\neq 1\right\},
\]
\[
f_X(n):=\max \{ f_X(w):w\ne1, \|w\|_{X\cup\{x,x^{-1}\}}\le n\},
\]
and
$$
\phi_X(n):= \min\{ \|\gc\|_X: w(\gamma)\neq 1 \quad \forall w\neq 1\text{ such that }\|w\|_{X\cup \left\{ x , x ^{-1} \right\}} \leq n\}.
$$
The functions $f_X(n),\phi_X(n)$ depend on $X$ but their growth types do not.

Bounds on the growth of $\phi_X$ of $\Gamma$ have important applications for the reduced $C^*$-algebra $C^*_r(\gC)$. For example, it was recently realized in \cite{Sri1} that if $\gC$ has rapid decay and $\phi_X(n)$ grows subexponentially then $C^*_r(\gC)$ is {\it selfless} in the sense of \cite{Selfless} and, hence, has strict comparison. Strict comparison, defined in \cite{Strict}, encodes certain regularity properties for K-theory and plays a significant role in the study of trace spaces.

Although the importance of strict comparison was widely understood for some time, prior to \cite{Sri1} not a single example of a finitely generated non-amenable group $\Gamma$ such that $C^*_r(\Gamma)$ has strict comparison was known. The breakthrough of \cite{Sri1} was to associate selflessness with the subexponential growth of $\phi_X$.
We also note that the fact that $\phi_X(n)^{1/n}\to 1$ for the free group $F_X$ (proved in \cite{Sri1})
was crucial for \cite{Sri2}, which established a negative solution of the $C^*$-algebraic Tarski problem, namely that for $n\ne m$, $C^*_r(F_n)$ and $C^*_r(F_m)$ are not elementarily equivalent (i.e. can be distinguished by first-order logic). More generally, it is proved in \cite{Sri1} that acylindrically hyperbolic groups with trivial amenable radicals and rapid decay are selfless. Recently, Vigdorovich \cite{Itamar} established analogous results with linear estimates (i.e., that $\phi_X(n)\le Cn$) for uniform lattices in $\SL_d(k)$ where $k$ is a local field. For $G=\SL_3(k)$ with $k$ archimedean, Lafforgue \cite{Strong(T)} proved that a uniform lattice $\gC\le G$ has rapid decay. As a consequence, $C_r^*(\gC)$ is selfless and has strict comparison.

Motivated by these developments, we study MIF for linear groups in general:

\begin{thm}\label{thm:main} 
Let $\gC\le GL(d,\BC)$ be a linear group generated by a finite set $X$. If $\gC$ is MIF then there is (a computable constant) $C=C(X)$ such that, for every $n\geq2$, \begin{enumerate}
\item $f_X(n) \leq C\log(n)$. 
\item $\phi_X(n)\le Cn$.
\end{enumerate} 
\end{thm}

The first claim extends results from \cite{thom}. Moreover, Bradford proved that every finitely generated MIF group satisfies
\[
f_X(n)\geq c_X\log n
\]
for all sufficiently large $n$, where $c_X>0$ depends on the generating set; see \cite[Theorem 1.9]{Bradford}. Thus the first claim is asymptotically sharp.

Groups that satisfy the first claim of Theorem \ref{thm:main} are called {\it sharply MIF}; groups that satisfy the second claim are called {\it linearly MIF}. Both notions were studied by various authors. However, we haven't found a reference for the following observation:

\begin{prop}\label{prop:Sha->lin}
Sharply MIF implies linearly MIF.
\end{prop}

Our proof of Theorem \ref{thm:main} is probabilistic, yet effective. We show that if $(\gamma_k)$ is the random walk on $\Gamma$ then, for a mixed word $w(x)$, the probability that $w(\gamma_k)=1$ decays exponentially in $k$, where the rate depends only on $X$. 
In an outstanding ongoing work \cite{BeBr}, Becker and Breuillard prove a stronger result about escaping general subvarieties with uniform estimates with respect to generating sets. 

For general groups, it will be interesting to know if random walk is always the reason behind sharp MIF. The following was conjectured by Assaf Naor.

\begin{conj}\label{conj:naor}
Let $\Gamma=\langle X\rangle$ be a finitely generated sharply MIF group, where $X$ is
finite and symmetric. Let $(\gamma_k)_{k\geq 0}$ be the simple random walk on $\Gamma$
generated by $X$ and started at the identity. Then there are constants $C,c>0$, depending
only on $X$, such that, for every nontrivial $w\in\Gamma*\langle x\rangle$ and every
$k\geq 1$,
\[
\Pr\bigl(w(\gamma_k)=1\bigr)
\leq C\|w\|_{X\cup \left\{ x,x^{-1} \right\}} e^{-ck}.
\]
\end{conj}


We shall prove this conjecture for linear groups, see Proposition \ref{prop:conj-linear}.

\subsection{Probabilistic super approximation}


To prove Theorem \ref{thm:main}, we estimate the probability that $w(\gamma_k)=1$ by reducing modulo an appropriately chosen prime and applying a variant of the super approximation theorem. Recall the super approximation theorem for groups with rational entries:

\begin{thm}\label{thm:CSA}
Let $X \subseteq \GL_d(\mathbb{Q})$ be a finite symmetric set, let $\Gamma := \langle X \rangle$ be the group generated by $X$, and let $G:= \overline{\Gamma}$. Assume that $G$ is connected, simply connected, and semisimple. There is a natural number $p_0$ such that if $p>p_0$ is a prime number then the reduction mod $p$ map $\varphi :\gC\to G(\BF_p)$ is surjective and the Cayley graphs $\Cay(G(\mathbb{F}_p),\varphi(X))$ form a family of expanders.
\end{thm}

We refer to \cite{GV} for a more general version (as well as to Theorem 1.2 in the excellent overview \cite{Br}). 
The assumption that the entries are rational (or at least algebraic) may appear crucial since the repulsive phenomenon between (algebraic) integers plays an important role in the spectral gap arguments. For linear groups with rational (or algebraic) entries, e.g. $S$-arithmetic groups or their thin subgroups, we could rely on Theorem \ref{thm:CSA} in our proof of Theorem \ref{thm:main}. 
For groups with nonalgebraic entries, one may obtain variants of Theorem \ref{thm:CSA} by considering small deformations with algebraic entries. However, these deformations are non-faithful, and the price (measured in heights of integers) of making them faithful on large balls is too expensive for our application. We instead prove and use a probabilistic variant of the super approximation theorem that works for matrices with arbitrary entries.  

For the statement, note that if $R$ is a finitely generated ring and $p$ is a prime then the set $\Hom(R,\mathbb{F}_p)$ is finite. If $k$ is a field, $\varphi\in\Hom(R,k)$, and $\mathbb{X}$ is an $R$-scheme, write
\[
\mathbb{X}^{\varphi}:=\mathbb{X}\times_{\Spec R,\varphi}\Spec k.
\]
We denote the induced map $\mathbb{X}(R)\to\mathbb{X}^{\varphi}(k)$ by $\varphi$; if $Y\subseteq\mathbb{X}(R)$, we write $Y^{\varphi}:=\varphi(Y)$.

\begin{theorem} \label{thm:prob.SSA} Let $R \subseteq \mathbb{C}$ be a finitely generated ring, let $X \subseteq \GL_d(R)$ be a finite symmetric set, and let $\Gamma = \langle X \rangle$. Assume that the Zariski closure of $\Gamma$ is connected, simply connected, and semisimple. Then there is an element $r\in R\smallsetminus\{0\}$, a group scheme $\mathbb{G} \subseteq \GL_d$ defined over $R$, and a finite Galois extension $F/ \mathbb{Q}$ such that \begin{enumerate}
\item \label{item:prob.SSA.R} $X \subseteq \mathbb{G}(R)$.
\item \label{item:prob.SSA.SA} For every prime number $p$ and every homomorphism $\varphi \in \Hom(R[r ^{-1}],\mathbb{F}_p)$, the algebraic group $\mathbb{G}^{\varphi}$ is connected, simply connected, and semisimple and $\varphi(\Gamma)=\mathbb{G} ^{\varphi}(\mathbb{F}_p)$.
\item \label{item:prob.SSA.Super} For every $\eta >0$ there are $p_0,\epsilon >0$ such that if $p>p_0$ is a prime number that splits completely in $F$ and $\varphi$ is a uniformly distributed element of $\Hom(R[r ^{-1}],\mathbb{F}_p)$, then 
\[
\Pr \left( \text{$\Cay \left( \mathbb{G} ^ \varphi(\mathbb{F}_p), \varphi(X) \right)$ is an $\epsilon$-expander} \right) > 1-\frac{1}{p^{1-\eta}}.
\]
\end{enumerate} 
\end{theorem} 

\begin{remark} Claims \ref{item:prob.SSA.R} and \ref{item:prob.SSA.SA} are special cases of the strong approximation theorem of \cite{Wei84}.
\end{remark} 

\paragraph{Acknowledgements} Thanks to Itamar Vigdorovich for discussions about MIF and for a careful reading of the first draft, to Emmanuel Breuillard, Sri Kunnawalkam Elayavalli, and Maks Radziwill for their help navigating the literature, to Amos Nevo for a stimulating discussion, and to Assaf Naor for suggesting Conjecture \ref{conj:naor}. We have assisted ChatGPT for the revision of the manuscript following the referees corrections and advices.

\section{Super approximation}

\subsection{Heights}

\begin{theorem} \label{thm:ht.exists}
Let $N\geq2$ and $t\geq0$, put
\[
R_0:=\mathbb{Z}\left[\frac1N\right][x_1,\ldots,x_t],
\]
and let $R$ be a domain which is finite as an $R_0$-module. Let $F_0$ be the algebraic closure of $\mathbb{Q}$ in $\Frac(R)$, and let $F/\mathbb{Q}$ be a finite Galois extension containing $F_0$. There are a function $\h_R:R\to\mathbb{R}_{\geq0}$ and a constant $C_R>1$ such that:
\begin{enumerate}
\item For every $r_1,\ldots,r_m\in R$,
\begin{enumerate}
\item \label{item:ht.+}
$\h_R(r_1+\cdots+r_m)\leq\log_N(m)+\max_i\h_R(r_i)$.
\item \label{item:ht.*}
$\h_R(r_1\cdots r_m)\leq mC_R+(t+1)\sum_i\h_R(r_i)$.
\end{enumerate}
\item \label{item:ht.mod.p}
If $0\neq r\in R$, $p$ splits completely in $F$, and $p>C_R^{\h_R(r)}$, then the image of $r$ in $R\otimes\mathbb{F}_p$ is not a zero divisor.
\item \label{item:ht.mod.p.2}
If $0\neq r\in R$ and $n\geq\max\{N,\h_R(r)\}$, there is a prime $p$ which splits completely in $F$ and satisfies
\[
n\leq p\leq C_Rn,
\]
such that the image of $r$ in $R\otimes\mathbb{F}_p$ is not a zero divisor.
\end{enumerate}
\end{theorem}

\subsubsection{Proof of Theorem \ref{thm:ht.exists} for $R=\mathbb{Z}[1/N]$}

By the Chebotarev prime number theorem (see, for example, \cite[p.~143]{IwKo04}), we may choose $C_0>\max\{N^2,2\}$ and $n_0>N$ such that, for every $n\geq n_0$,
\begin{equation}\label{eq:primorial}
\prod_{\substack{n\leq p\leq C_0n\\ p\text{ splits completely in }F}}p>N^{2n}.
\end{equation}
Set $\h_{\mathbb{Z}[1/N]}(0)=0$. If $a/N^k\neq0$, where $k\geq0$ is minimal (equivalently, either $k=0$ or $N\nmid a$), set
\[
\h_{\mathbb{Z}[1/N]}\left(\frac{a}{N^k}\right)
=n_0+\max\left\{k,\log_N^+\left|\frac{a}{N^k}\right|\right\}.
\]
The additive constant $n_0$ ensures that \eqref{eq:primorial} is available whenever $n\geq\h_{\mathbb{Z}[1/N]}(r)$.

For $r_1,\ldots,r_m\in\mathbb{Z}[1/N]$ one has
\begin{equation}\label{eq:ht.Z.+}
\h_{\mathbb{Z}[1/N]}(r_1+\cdots+r_m)
\leq\log_N(m)+\max_i\h_{\mathbb{Z}[1/N]}(r_i)
\end{equation}
and
\begin{equation}\label{eq:ht.Z.*}
\h_{\mathbb{Z}[1/N]}(r_1r_2)
\leq\h_{\mathbb{Z}[1/N]}(r_1)+\h_{\mathbb{Z}[1/N]}(r_2).
\end{equation}
This proves Claims \ref{item:ht.+} and \ref{item:ht.*} in this case.

Let $r=a/N^k\neq0$, with $k$ minimal as above. The definition gives
\[
\log_N|a|\leq2\h_{\mathbb{Z}[1/N]}(r).
\]
If $p>C_0^{\h_{\mathbb{Z}[1/N]}(r)}$, then $p>N$ and, because $C_0\geq N^2$, one has $p>|a|$. Thus the reduction of $r$ modulo $p$ is nonzero. This proves Claim \ref{item:ht.mod.p}.

Finally, suppose that $n\geq\max\{N,\h_{\mathbb{Z}[1/N]}(r)\}$. Then $|a|\leq N^{2n}$. By \eqref{eq:primorial}, not every prime in $[n,C_0n]$ which splits completely in $F$ can divide $a$. Such a prime at which $a$ remains nonzero proves Claim \ref{item:ht.mod.p.2}.

\subsubsection{Proof of Theorem \ref{thm:ht.exists} for $R=R_0$}

Set $\h_{R_0}(0)=0$. For a nonzero polynomial $a(x)=\sum_I a_Ix^I$ define
\[
\h_{R_0}\left(\sum_Ia_Ix^I\right)
=\max_{a_I\neq0}\left\{\h_{\mathbb{Z}[1/N]}(a_I)+|I|\right\},
\qquad |I|:=\sum_{j=1}^t I_j.
\]
Claims \ref{item:ht.+}, \ref{item:ht.mod.p}, and \ref{item:ht.mod.p.2} follow from the corresponding assertions for $\mathbb{Z}[1/N]$, by considering any nonzero coefficient and using that $\mathbb{F}_p[x_1,\ldots,x_t]$ is a domain.

For $0\neq a(x)=\sum_Ia_Ix^I$, write
\[
\supp(a):=\{I:a_I\neq0\},
\qquad
\deg(a):=\max\{|I|:I\in\supp(a)\}.
\]
Then
\begin{equation}\label{eq:support.R0}
|\supp(a)|\leq\binom{\deg(a)+t}{t}\leq(\deg(a)+1)^t.
\end{equation}
If $c(x)=a(x)b(x)$ and $I\in\supp(c)$, then
\[
\begin{split}
\h_{\mathbb{Z}[1/N]}(c_I)
&=\h_{\mathbb{Z}[1/N]}\left(\sum_{J\leq I}a_Jb_{I-J}\right)\\
&\leq\log_N|\supp(a)|+\max_J\h_{\mathbb{Z}[1/N]}(a_Jb_{I-J})\\
&\leq t\log_N(\deg(a)+1)+\h_{R_0}(a)+\h_{R_0}(b)-|I|\\
&\leq(t+1)\h_{R_0}(a)+\h_{R_0}(b)-|I|.
\end{split}
\]
Here the first inequality uses \eqref{eq:ht.Z.+}; the second uses \eqref{eq:ht.Z.*}, \eqref{eq:support.R0}, and the estimates
\[
\h_{\mathbb{Z}[1/N]}(a_J)\leq\h_{R_0}(a)-|J|,
\qquad
\h_{\mathbb{Z}[1/N]}(b_{I-J})\leq\h_{R_0}(b)-|I-J|.
\]
Consequently,
\begin{equation}\label{eq:ht.R_0.*.2}
\h_{R_0}(ab)\leq(t+1)\h_{R_0}(a)+\h_{R_0}(b),
\end{equation}
and Claim \ref{item:ht.*} follows by induction.

\subsubsection{Proof of Theorem \ref{thm:ht.exists} in general}

Choose generators $\alpha_1,\ldots,\alpha_s$ for $R$ as an $R_0$-module, put $\h_R(0)=0$, and, for $r\neq0$, define
\[
\h_R(r):=\min_{r=\sum_i a_i\alpha_i}\max_i\h_{R_0}(a_i).
\]
Claim \ref{item:ht.+} is immediate. Choose $c_{ij}^k\in R_0$ such that
$\alpha_i\alpha_j=\sum_kc_{ij}^k\alpha_k$. Applying \eqref{eq:ht.Z.+} and \eqref{eq:ht.R_0.*.2} to expressions realizing the two heights gives a constant $C_1$ such that
\[
\h_R(r_1r_2)\leq C_1+(t+1)\h_R(r_1)+\h_R(r_2).
\]
Taking $C_R$ to be large enough, Claim \ref{item:ht.*} follows by induction.

Let $K=\Frac(R)$ and $K_0=\mathbb{Q}(x_1,\ldots,x_t)$. Since $R$ is integral over the integrally closed domain $R_0$, the field norm restricts to a map
\[
\mathcal N_{K/K_0}:R\longrightarrow R_0.
\]
The norm is a polynomial, of bounded degree, in the $R_0$-coordinates of an element of $R$. The preceding height inequalities therefore give a constant $A\geq1$ such that
\begin{equation}\label{eq:height.norm}
\h_{R_0}\bigl(\mathcal N_{K/K_0}(r)\bigr)\leq A\h_R(r)
\qquad(r\in R).
\end{equation}

Because $F$ contains the algebraic closure $F_0$ of $\mathbb Q$ in $K$, the irreducible components of $\Spec(R\otimes_{\mathbb Z}F)$ are geometrically integral. Choosing integral models of these components and using the openness of geometric integrality \cite[Proposition 9.7.8]{EGA_IV_3}, there is a finite set $\Sigma$ of rational primes such that, whenever $p\notin\Sigma$ splits completely in $F$, the ring $R\otimes\mathbb F_p$ is reduced and its irreducible components are the reductions of those components. After enlarging $\Sigma$ if necessary, the determinant defining the finite field norm commutes with reduction modulo $p$. Consequently, on the function field of each irreducible component, the reduction of $\mathcal N_{K/K_0}(r)$ is the product of the corresponding field norms of the reductions of $r$.

Let $0\neq r\in R$ and let $p\notin\Sigma$ split completely in $F$. If the reduction of $\mathcal N_{K/K_0}(r)$ in $R_0\otimes\mathbb F_p$ is nonzero, then the reduction of $r$ is nonzero at the generic point of every irreducible component of $\Spec(R\otimes\mathbb F_p)$. Since this fiber is reduced, the reduction of $r$ is not a zero divisor.

By \eqref{eq:height.norm} and Claim \ref{item:ht.mod.p} for $R_0$, this applies whenever
$p>C_2^{\h_R(r)}$, after choosing $C_2$ sufficiently large to exclude $\Sigma$. This proves Claim \ref{item:ht.mod.p} for $R$.

A similar argument proves Claim \ref{item:ht.mod.p.2}.

\subsection{Random homomorphisms}

\begin{lemma} \label{lem:size.Hom}
Let $R\subseteq\mathbb C$ be a finitely generated domain, let $t=\operatorname{trdeg}_{\mathbb Q}\Frac(R)$, and let $F_0$ be the algebraic closure of $\mathbb Q$ in $\Frac(R)$. Let $F/\mathbb Q$ be a finite Galois extension containing $F_0$. There is a constant $C$ such that, if $p>C$ splits completely in $F$, then
\[
\left|\,|\Hom(R,\mathbb F_p)|-[F_0:\mathbb Q]p^t\right|<Cp^{t-1/2}.
\]
\end{lemma}

\begin{proof}
Put $\mathbb V=\Spec R$. The generic fiber $\mathbb V_{\mathbb Q}$ is integral of dimension $t$. Since $F_0$ is algebraically closed in $\Frac(R)$, its base change to $F$ has exactly $[F_0:\mathbb Q]$ geometrically integral components, one for each embedding of $F_0$ into $F$; see \cite[Corollary 4.5.10]{EGA_IV_2}.

Let $\mathcal O_F$ be the ring of integers of $F$. After inverting one integer, the closures of these components in $\mathbb V_{\mathcal O_F}$ cover $\mathbb V_{\mathcal O_F}$, each has geometrically integral fibers, and every pairwise intersection has fibers of dimension at most $t-1$. The assertion about geometric integrality follows from \cite[Proposition 9.7.8]{EGA_IV_3}; the remaining assertions follow by spreading out the corresponding statements on the generic fiber.

If $p$ is sufficiently large and splits completely in $F$, choose $\mathfrak p\mid p$ in $\mathcal O_F$, so $\mathcal O_F/\mathfrak p=\mathbb F_p$. Lang--Weil gives $p^t+O(p^{t-1/2})$ points on each of the $[F_0:\mathbb Q]$ components, while their pairwise intersections contribute $O(p^{t-1})$ points. Hence
\[
|\mathbb V(\mathbb F_p)|=[F_0:\mathbb Q]p^t+O(p^{t-1/2}),
\]
with a uniform implied constant. Since $\mathbb V(\mathbb F_p)=\Hom(R,\mathbb F_p)$, the lemma follows.
\end{proof}

\begin{lemma} \label{lem:DKL}
(\cite[Claim 7.2]{DKL14}) Let $\mathbb F$ be a finite field and let $V$ be an affine variety of dimension $k$ and degree $d$ defined over $\mathbb F$. Then $|V(\mathbb F)|\leq d|\mathbb F|^k$.
\end{lemma}

\begin{lemma} \label{lem:random.specialization.nonzero}
Let $R$ be a domain finite over $\mathbb Z[1/N][x_1,\ldots,x_t]$, let $\h_R$ be a height function supplied by Theorem \ref{thm:ht.exists}, and let $F/\mathbb Q$ be a finite Galois extension containing the algebraic closure of $\mathbb Q$ in $\Frac(R)$. There is a constant $C$ such that the following holds. Suppose that $p>C$ splits completely in $F$, the image of $r\in R$ in $R\otimes\mathbb F_p$ is not a zero divisor, and $\varphi$ is uniformly distributed in $\Hom(R,\mathbb F_p)$. Then
\[
\Pr(\varphi(r)=0)<\frac{C\h_R(r)}p.
\]
\end{lemma}

\begin{proof}
Put $\mathbb V=\Spec R$ and fix an affine embedding $\mathbb V\hookrightarrow\mathbb A^M$. The degrees of the fibers $\mathbb V_{\mathbb F_p}$ are uniformly bounded by some constant $C$. After possibly increasing $C$, every $r\in R$ is the restriction of a polynomial of degree at most $C\h_R(r)$. Thus, if $\mathbb V_r$ is the zero locus of $r$, B\'ezout's theorem gives
\[
\deg((\mathbb V_r)_{\mathbb F_p}) \leq \deg \mathbb{V}_{\mathbb{F}_p} \cdot \deg r\leq C^2\h_R(r).
\]
If $t=0$, the non-zero-divisor $r$ is a unit in the finite reduced algebra $R\otimes\mathbb F_p$, so the required probability is zero. Assume $t>0$. Since $r$ is not a zero divisor, it vanishes identically on no irreducible component of $\mathbb V_{\mathbb F_p}$; hence $\dim(\mathbb V_r)_{\mathbb F_p}\leq t-1$. Lemma \ref{lem:DKL} yields
\[
|\mathbb V_r(\mathbb F_p)|\leq C^2\h_R(r)p^{t-1}.
\]
On the other hand, Lemma \ref{lem:size.Hom} gives $|\mathbb V(\mathbb F_p)|\geq c p^t$ for all sufficiently large split primes. The result follows after increasing $C$ once more.
\end{proof}

\subsection{Generating pairs}

\begin{lemma} \label{lem:size.G} (\cite[Lemma 3.5]{Nor87}) If $G$ is a connected linear algebraic group defined over a finite field $\mathbb{F}$, then $\left( | \mathbb{F} |-1 \right)^{\dim G} \leq |G(\mathbb{F})| \leq \left( | \mathbb{F} |+1 \right)^{\dim G}$.
\end{lemma} 

\begin{lemma} \label{lem:V} Let $R \subseteq \mathbb{C}$ be a finitely generated ring and let $\mathbb{G}$ be a group scheme over $R$ such that $\mathbb{G}(\mathbb{C})$ is connected, simply connected, and semisimple. There is an element $r\in R \smallsetminus\{0\}$ and a closed subscheme
\[
\mathbb{V}\subseteq\mathbb{G}\times_{\Spec R}\mathbb{G}
\]
such that:
\begin{enumerate}
\item For every field $k$ and every homomorphism $\varphi:R[r^{-1}]\to k$, the algebraic group $\mathbb{G}^{\varphi}$ is connected, simply connected, and semisimple.
\item For every field $k$ and every homomorphism $\varphi:R[r^{-1}]\to k$, one has $\mathbb{V}^{\varphi}\neq(\mathbb{G}^{\varphi})^2$.
\item If $(g,h)\in\mathbb{G}(\mathbb{C})^2\smallsetminus\mathbb{V}(\mathbb{C})$, then $g,h$ generate a Zariski-dense subgroup of $\mathbb{G}(\mathbb{C})$.
\item For every prime $p$ and every homomorphism $\varphi:R[r^{-1}]\to\mathbb{F}_p$, if $(g,h)\in\mathbb{G}^{\varphi}(\mathbb{F}_p)^2\smallsetminus\mathbb{V}^{\varphi}(\mathbb{F}_p)$, then $g,h$ generate $\mathbb{G}^{\varphi}(\mathbb{F}_p)$.
\end{enumerate}
\end{lemma}

\begin{proof} Since the generic fiber of $\mathbb{G}$ is smooth, we can localize $R$ so that the map $\mathbb{G} \rightarrow \Spec R$ is smooth. \cite[Propositions 3.1.9 and 3.1.12]{Con} imply that there is a further localization of $R$ after which all fibers of $\mathbb{G} \rightarrow \Spec R$ are connected and semisimple.

By \cite[Lemma 5.1.3]{Con}, there is an etale cover $T \rightarrow \Spec R$ such that $\mathbb{G} \times_{\Spec R} T$ is split with constant root datum $\Psi(\mathbb{G},T)$. By the assumption, the root datum for $\mathbb{G}(\mathbb{C})$ is simply connected. Since $\Spec R$ is connected, it follows that the root datum for every fiber of $\mathbb{G}$ is simply connected. This proves the first assertion. 

Since $\Psi(\mathbb{G},T)$ is simply connected, it decomposes into as a direct product $\Psi(\mathbb{G},T)=\prod_{i\in I} \Psi_i$ of simply connected root data whose underlying root systems are irreducible. Letting $\mathbb{G}_i \rightarrow T$ be the corresponding split group schemes, we have
\[
\mathbb{G} \times_{\Spec R}T = \prod_{i\in I} \mathbb{G}_i.
\]
Note that the descent data coming from $T \rightarrow \Spec R$ permutes the $\mathbb{G}_i$.

Define a closed subscheme
\[
\mathbb{U}_T:=\bigcup_{i\in I} \prod_{j\in I \smallsetminus i} \mathbb{G}_i \subseteq \mathbb{G} \times_{\Spec R}T.
\]
Since the descent data preserves $\mathbb{U}_T$, faithfully flat descent implies that there is a closed subscheme $\mathbb{U} \subseteq \mathbb{G}$ such that $\mathbb{U}_T=\mathbb{U} \times_{\Spec R}T$. For every algebraically closed field $k$ and every $k$-point of $S$, the fiber $\mathbb U_k$ is the union of the maximal proper direct factors of $\mathbb G_k$, and hence contains every proper direct factor.

Fix a faithful embedding $\mathbb{G} \subseteq \GL_d$. Its restriction to every $\mathbb{G}_i$ is faithful. By Jordan's theorem, there is a constant $C_J$ such that every finite subgroup of $\GL_d(\mathbb{C})$ has an abelian subgroup of index at most $C_J$. By \cite{dense}, there are $g_1,g_2\in\mathbb G(\mathbb C)$ that generate a Zariski-dense free subgroup whose projection to every factor $\mathbb{G}_i(\mathbb{C})$ is injective. Since $\langle g_1,g_2 \rangle$ is Zariski dense, the set $\left\{ \Ad(h) \mid h\in \langle g_1,g_2 \rangle \right\}$ spans $\bigoplus \End(\Lie(\mathbb{G}_i(\mathbb{C})))$. This means that, denoting $N:=\sum_{i\in I} \dim(\mathbb{G}_i(\mathbb{C}))^2$, there are words $w_1,\ldots,w_N$ such that $\Ad(w_1(g_1,g_2),\ldots,\Ad(w_N(g_1,g_2)$ are linearly independent. Since the projections on $\mathbb{G}_i(\mathbb{C})$ are injective, the element $\left[ g_1^{C_J!},g_2^{C_J!} \right]$ does not belong to $\mathbb{U}(\mathbb{C})$.

Define $\mathbb{V} \subseteq \mathbb{G} \times_{\Spec R} \mathbb{G}$ as the closed subscheme of pairs $(x,y)$ satisfying at least one of the following conditions: \begin{enumerate}
\item \label{condition 1 of V} $\Ad(w_1(x,y)),\ldots,\Ad(w_N(x,y))$ are linearly dependent.
\item \label{condition 2 of V} $\left[ x^{C_J!},y^{C_J!} \right]\in \mathbb{U}$
\end{enumerate} 
Let $\mathbb{W}$ be the scheme $\left( \mathbb{G} \times_{\Spec R} \mathbb{G} \right) \smallsetminus \mathbb{V}$ and let $q: \mathbb{G} \times_{\Spec R} \mathbb{G} \rightarrow \Spec R$ be the structure map. Since $q$ is smooth, it is open, so $q(\mathbb{W})$ is open. By construction, $(g_1,g_2)\in \mathbb{W}(\mathbb{C})$, so $\mathbb{W}$ is nonempty. Thus, after some localization of $R$, $q(\mathbb{W})=\Spec R$ and assertion 2 holds.

For assertion 3, let $(g,h)\in \mathbb{G}(\mathbb{C})^2 \smallsetminus \mathbb{V}(\mathbb{C})$. Let $H$ be the Zariski closure of $\langle g,h \rangle$ and let $\mathfrak{h}$ be the Lie algebra of $H$. Since $(g,h)$ does not satisfy condition \eqref{condition 2 of V}, $\mathfrak{h}$ projects nontrivially on every simple factor. Since $\mathfrak{h}$ is invariant under $\Ad(\langle g,h \rangle)$ and $(g,h)$ does not satisfy condition \eqref{condition 1 of V}, $\mathfrak{h}=\Lie(\mathbb{G}(\mathbb{C}))$, so $H=\mathbb{G}(\mathbb{C})$.

It remains to prove assertion 4. For this, it is enough to show that there is a constant $C=C(d)$ such that assertion 4 holds for $p>C$ (because we can localize in the product of all primes less than or equal to $C$).

Suppose that $p$ is a prime, that $\varphi :R \rightarrow \mathbb{F}_p$, and that $(g,h)\in \mathbb{G} ^ \varphi (\mathbb{F}_p)^2 \smallsetminus V^ \varphi (\mathbb{F}_p)$. 

We first require $C(d)>d$ (so $p>d$). Since $\mathbb{G} ^ \varphi $ is semisimple, if $z\in \Mat_d(\mathbb{F}_p)$ is nilpotent then $z\in \Lie(\mathbb{G})$ if and only if $\exp(z)\in \mathbb{G} ^ \varphi(\mathbb{F}_p)$. Since $\mathbb{G} ^ \varphi $ is simply connected, it decomposes as $\mathbb{G} ^ \varphi = H_1 \times \cdots H_m$, where $H_i$ are simple groups over $\mathbb{F}_p$ (in general, not absolutely simple). Let $\pi_i:\mathbb{G} ^ \varphi \rightarrow H_i$ be the projection. 

We claim that, for every $i$, the order of $\pi_i(\langle g,h \rangle)$ is divisible by $p$. Indeed, if the order were prime to $p$ then $H^2(\pi_i(\langle g,h \rangle),V)=0$ for every finite $p$-torsion module $V$. Working inductively, there would be a lift of $\pi_i(\langle g,h \rangle)$ to a subgroup of $\GL_d(\mathbb{Z}/p^n)$ for every $n$ and thus a lift to $\GL_d(\mathbb{Z}_p) \subseteq \GL_d(\mathbb{C})$. This would have implied that $\left[ \pi_i(g)^{C_J!},\pi_i(h)^{C_J!} \right] =1$, a contradiction since $(g,h)$ does not satisfy condition \eqref{condition 2 of V}. Thus, for every $i$ there is an element $u_i\in \langle g,h \rangle$ such that $\pi_i(u_i)$ is a $p$-element. Raising $u_i$ to a prime-to-$p$ power, we can assume that $u_i$ is a $p$-element itself. Since $p>d$, the element $\log(u_i)$ is defined, belongs to $\Lie(\mathbb{G} ^ \varphi)$, and has a nonzero projection to $\Lie(H_i)$.

Denote $D=\dim \mathbb{G} ^ \varphi $. Since $(g,h)$ does not satisfy condition \eqref{condition 1 of V}, there are elements $\gamma_1,\ldots,\gamma_D\in \langle g,h \rangle$ and indexes $i_1,\ldots,i_D$ such that the elements $\gamma_j \log(u_{i_j}) \gamma_j ^{-1} = \log(\gamma_j u_{i_j} \gamma_j ^{-1})$ are linearly independent. The map $\mathbb{F}_p^D \rightarrow \langle g,h \rangle$ given by
\[
(a_1,\ldots,a_D)\mapsto \left( \gamma_1 u_{i_1}^{a_1} \gamma_1 ^{-1} \right) \cdots \left( \gamma_D u_{i_D}^{a_D} \gamma_D ^{-1} \right)
\]
is the restriction of a polynomial map $\Phi : \mathbb{A} ^D \rightarrow \mathbb{G} ^ \varphi$ whose degree is at most $d^2$ and whose derivative at $0$ is an isomorphism. By \cite[Proof of Theorem B]{Nor87}, there is a constant $C_0=C_0(d)$ such that $\left| \Phi(\mathbb{F}_p^D) \right| > \frac 1{C_0}p^D$. Since $\Phi(\mathbb{F}_p) \subseteq \langle g,h \rangle$ and by Lemma \ref{lem:size.G}, the index of $\langle g,h \rangle $ in $\mathbb{G} ^ \varphi(\mathbb{F}_p)$ is bounded. As $p \rightarrow \infty$, the minimal index of a nontrivial subgroup of $\mathbb{G} ^ \varphi(\mathbb{F}_p)$ tends to infinity, so there is $C(d)$ such that $p>C(d)$ implies that $\langle g,h \rangle=\mathbb{G} ^ \varphi(\mathbb{F}_p)$.
\end{proof}

\subsection{Walks on semisimple groups}

\begin{lemma} \label{lem:EMO} Let $G$ be a connected semisimple algebraic group over $\mathbb{C}$ and let $U \subseteq G$ be a proper subvariety. There is a constant $C_{EMO}$ such that if $Y \subseteq G(\mathbb{C})$ is a set that generates a Zariski dense subgroup, then $Y^{C_{EMO}}$ is not contained in $U$.
\end{lemma} 

\begin{proof}
This follows from \cite[Lemma 2.5]{uti}. The proof is essentially the argument of \cite[Lemma 3.2]{EMO}.
\end{proof} 

We will also make use of the following consequence of Breuillard's non-concentration theorem.\footnote{In the first version appeared in the arXiv we gave an alternative proof for a variant of this lemma. In that proof we deduced the lemma from the classical super approximation theorem.}

\begin{lemma}\label{lem:quant-nonconcentration}
Let $G$ be a linear algebraic group over $\mathbb{C}$ whose identity component is semisimple,
and let $X\subseteq G(\mathbb{C})$ be a finite symmetric set generating a Zariski-dense subgroup. Then there are constants $c,C>0$ and $A\geq 1$,
depending only on $G$ and $X$, such that, for every closed subvariety $U\subseteq G$
which contains no connected component of $G$ and every $n\geq 1$,
\[
\mu_X^{*n}\bigl(U(\mathbb{C})\bigr)\leq C(1+\deg U)^A e^{-cn}.
\]
\end{lemma}

\begin{proof}
Choose a faithful representation $G\hookrightarrow\GL_d$ and put $\Gamma=\langle
X\rangle$. If $G^\circ$ is trivial the assertion is immediate. Otherwise, $\Gamma$ is not virtually solvable,
the solvable radical of $G$ has dimension zero, and $\dim U<\dim G$.
Theorem~7.2 of \cite{Br} therefore gives exponential non-concentration in
$U$. In its proof, after reductions depending only on $G$ and $X$, the dependence on the
subvariety enters through a Schwarz--Zippel estimate. In fixed ambient dimension this is
bounded by a polynomial in $\deg U$, which gives the displayed estimate.
\end{proof}

\begin{lemma} \label{lem:RW.general} Let $(V,E)$ be a regular $\epsilon$-expander graph, for some constant $\epsilon$, and let $v_0\in V$. Let $(x_k)$ be the simple random walk on $(V,E)$ that starts at $v_0$. For every subset $U \subseteq V$ and every $k \geq -\frac{\log|V|}{\log(1- \epsilon)}$,
\[
\Pr \left( x_k\in U\right) < 2\frac{|U|}{|V|}.
\]
\end{lemma} 

\begin{proof}
Let $M\in\End(\mathbb C^V)$ be the Markov operator corresponding to the random walk, and let $1=\lambda_1,\lambda_2,\ldots,\lambda_{|V|}$ be its eigenvalues. By assumption, $|\lambda_i|<1-\epsilon$ for $i=2,\ldots,|V|$. Let $1_{v_0},1_U,1_V\in\mathbb C^V$ be the indicator functions of $\{v_0\}$, $U$, and $V$, respectively. Write
\[
1_{v_0}=\frac1{|V|}1_V+\phi,
\]
where $\phi$ is in the span of the eigenvectors corresponding to $\lambda_2,\ldots,\lambda_{|V|}$. Then
\begin{align*}
\Pr(x_k\in U)
&=\langle M^k1_{v_0},1_U\rangle\\
&=\left\langle\frac1{|V|}1_V,1_U\right\rangle
  +\langle M^k\phi,1_U\rangle\\
&\leq\frac{|U|}{|V|}+(1-\epsilon)^k\,\|\phi\|\,\|1_U\|\\
&\leq\frac{|U|+\sqrt{|U|}}{|V|}.
\end{align*}
If $U$ is nonempty, then $\sqrt{|U|}\leq|U|$; the empty case is immediate. The result follows.
\end{proof}

We finish this subsection by verifying Conjecture \ref{conj:naor} for linear groups.

\begin{proposition}\label{prop:conj-linear}
Conjecture~\ref{conj:naor} holds for finitely generated linear groups.
\end{proposition}

\begin{proof}
By Proposition~\ref{prop:3.1} below, we may choose a faithful
representation
\[
\rho:\Gamma\longrightarrow\GL_m(\mathbb C)
\]
for which the identity component of the Zariski closure
$G:=\overline{\rho(\Gamma)}^Z$ is semisimple. We identify $\Gamma$ with its image under $\rho$
and put
\[
\Gamma^\circ:=\Gamma\cap G^\circ,
\qquad
q:=[\Gamma:\Gamma^\circ].
\]
Fix a nontrivial mixed word $w$ and set
\[
L:=\|w\|_{X\cup\{x,x^{-1}\}}.
\]
If $w$ does not involve $x$ there is nothing to prove. We therefore assume that $w$ involves $x$.

We first claim that the regular map $g\mapsto w(g)$ is not the constant $I$ on any connected component of $G$. This follows because, for every $a\in \Gamma$, the word $w(ax^q)$ is non-trivial.

Let $h_{ij}$ be the matrix entries of the regular map $w-I$ on $G$. For every connected
component $D$ of $G$, not all of the restrictions $h_{ij}|_D$ vanish identically.
Therefore one may choose rational numbers $\lambda_{ij}$ such that
\[
h_w:=\sum_{i,j}\lambda_{ij}h_{ij}
\]
is nonzero on every connected component of $G$. $h_w$ is the restriction of a polynomial $H_w$. The zero set
\[
V_w:=\{g\in\GL_m:H_w(g)=0\}
\]
contains every solution of $w(g)=1$, contains no connected component of $G$, and
satisfies
\begin{equation}\label{eq:word-variety-degree}
\deg(V_w)\leq C_0L
\end{equation}
for a constant $C_0$ depending only on the fixed representation. 

Put $U_w:=G\cap V_w$. By B\'{e}zout's theorem and
\eqref{eq:word-variety-degree}, $\deg U_w\leq C_1L$ for a constant $C_1$ depending only
on $G$ and the fixed representation. Since $U_w$ contains no connected component of
$G$, Lemma~\ref{lem:quant-nonconcentration} gives constants $c,C>0$ and $A\geq1$,
depending only on $G$ and $X$, such that
\begin{equation}\label{eq:prop210-polynomial}
\Pr(\gamma_k\in U_w)
\leq C(1+\deg U_w)^Ae^{-ck}
\leq C_2L^Ae^{-ck}.
\end{equation}
Since $w(\gamma_k)=1$ implies $\gamma_k\in U_w$, the same bound holds for the
probability in the statement.

It remains to put \eqref{eq:prop210-polynomial} in the asserted form. Set
\[
c':=\frac{c}{2A},
\qquad
t:=Le^{-c'k}.
\]
If $t\geq1$, then $\Pr(w(\gamma_k)=1)\leq1\leq t$. If $t<1$, then
\[
C_2L^Ae^{-ck}=C_2\frac{t^{2A}}{L^A}\leq C_2t.
\]
Thus, for every $k\geq1$,
\[
\Pr\bigl(w(\gamma_k)=1\bigr)
\leq\max\{1,C_2\}Le^{-c'k},
\]
which proves the proposition.
\end{proof}

\subsection{A criterion for expansion}

The following expansion criterion combines results of Bourgain--Gamburd, Breuillard--Green--Tao, Pyber--Szabo, and Landazuri--Seitz.

\begin{theorem} \label{thm:Bourgain.Gamburd} For every $r\in \mathbb{N}$ and $\delta >0$ there is $\epsilon >0$ such that the following holds: if $p$ is a prime number, $G$ is a simply connected semisimple group over $\mathbb{F}_p$ of rank at most $r$, $X \subseteq G(\mathbb{F}_p)$ is a symmetric generating set, and there is an integer $n < \log p$ such that, for all proper subgroups $H \subsetneq G(\mathbb{F}_p)$, $\mu_X^n(H)<p^ {-\delta}$, then the first nonzero eigenvalue of $\Cay(G(\mathbb{F}_p),X)$ is greater than $\epsilon$.
\end{theorem} 

\begin{remark}
We use the expansion criterion in the form of \cite[Proposition 3.1]{Br}. Its first hypothesis (quasirandomness) follows from \cite[Proposition 6.1]{Br}: the group $G(\mathbb{F}_p)$ is an almost direct product of a bounded number of finite groups of Lie type, and the Landazuri--Seitz bounds \cite{LaSe74} give the required lower bound on dimensions of nontrivial representations. The second hypothesis, the classification of approximate subgroups, follows from \cite{BGT11} or \cite{PySa16}. The third hypothesis is precisely the non-concentration assumption appearing in Theorem \ref{thm:Bourgain.Gamburd}.
\end{remark} 

\subsection{Proof of Theorem \ref{thm:prob.SSA}}

Let $R,X,\Gamma$ be as in the theorem. Denote the fraction field of $R$ by $K$ and the Zariski closure of $\Gamma$ by $G$. Being invariant under $\Gal(\mathbb{C} / K)$, the algebraic group $G$ is defined over $K$. Choose a group scheme $\mathbb{G} \subseteq \GL_d$ defined over $R$ such that $\mathbb{G} \times \Spec K=G$. Since $X \subseteq G \cap \GL_d(R)=\mathbb{G}(R)$, Claim \ref{item:prob.SSA.R} of Theorem \ref{thm:prob.SSA} holds.

Let $r_1\in R \smallsetminus\{0\}$ and $\mathbb{V} \subseteq \mathbb{G} \times_{\Spec R} \mathbb{G}$ be obtained by applying Lemma \ref{lem:V} to $R$ and $\mathbb{G}$. Choose polynomials $f_1,\ldots,f_M$ that generate the ideal of $\mathbb{V}$. Since $\Gamma \times \Gamma$ is Zariski dense in $\mathbb{G}(\mathbb{C})^2$, there are elements $g_1,g_2\in \Gamma$ such that $(g_1,g_2)\notin \mathbb{V}(\mathbb{C})$. Consequently, there is $1 \leq i \leq M$ such that $f_i(g_1,g_2)\neq 0$.

We show that Claim \ref{item:prob.SSA.SA} of Theorem \ref{thm:prob.SSA} holds for every element $r\in R \smallsetminus\{0\}$ that is divisible by $r_1 \cdot f_i(g_1,g_2)$. Indeed, let $\varphi: R[r ^{-1}] \rightarrow \mathbb{F}_p$. The first claim of Lemma \ref{lem:V} says that $\mathbb{G}^{\varphi}$ is connected, simply connected, and semisimple. In addition, $\varphi(f_i(g_1,g_2))\neq 0$, so $(\varphi(g_1),\varphi(g_2))\notin \mathbb{V} ^ \varphi(\mathbb{F}_p)$. The fourth claim of Lemma \ref{lem:V} implies that $\varphi (\Gamma) = \mathbb{G} ^ \varphi(\mathbb{F}_p)$.

Let $t$ be the transcendence degree of $K/ \mathbb{Q}$. By Noether's normalization theorem, there is an integer $N\geq2$ such that $R[(Nr_1f_i(g_1,g_2)) ^{-1}]$ is a finite extension of $\mathbb{Z}[1/N][x_1,\ldots,x_t]$. Let $F_0$ be the algebraic closure of $\mathbb{Q}$ in $\Frac(R)$, and choose a finite Galois extension $F/\mathbb Q$ containing $F_0$. We will show that Claim \ref{item:prob.SSA.Super} of Theorem \ref{thm:prob.SSA} holds for this $F$ and for $r=Nr_1f_i(g_1,g_2)$. 

Fix $\eta >0$. By Claims \ref{item:prob.SSA.R}, \ref{item:prob.SSA.SA}, and Theorem \ref{thm:Bourgain.Gamburd}, it is enough to show that there is $\delta >0$ with the following property: if $p$ is a large enough prime number that splits completely in $F$ then there is a number $m<\log p$ such that if $\varphi$ is a uniformly distributed random element of $\Hom(R[r ^{-1}],\mathbb{F}_p)$ then
\begin{equation} \label{eq:RW.Delta}
\Pr \left( \text{$\mu_{\varphi(X)}^{*m}(\Delta)<p^{-\delta}$ for every proper subgroup $\Delta \subsetneq \mathbb{G} ^ \varphi(\mathbb{F}_p)$}\right) > 1-\frac{1}{p^{1- \eta}}.
\end{equation}
Denote $S:=R[r ^{-1}]$ and let $\h_S$ be the height function obtained by applying Theorem \ref{thm:ht.exists} to $S$ and $F$. We extend $\h_S$ to $\mathbb{G}(S)$ by
\[
\h_S((a_{i,j})):=\max_{i,j}\h_S(a_{i,j}).
\]
After increasing a constant $C>1$, we may assume that it is larger than the constants in Theorem \ref{thm:ht.exists} and Lemma \ref{lem:random.specialization.nonzero} and that:
\begin{enumerate}
\item \label{item:ht.X} For every $m$ and every $g\in X^m$,
\[
\h_S(g)<Cm.
\]
\item \label{item:ht.V} For every $i\in\{1,\ldots,M\}$ and every $g,h\in\mathbb G(S)$,
\[
\h_S(f_i(g,h))<C\bigl(1+\max\{\h_S(g),\h_S(h)\}\bigr).
\]
\end{enumerate}

Let $0<\alpha<1/2$ be a small constant to be set later. Suppose that $p$ is a prime that splits completely in $F$ and consider a uniformly distributed random homomorphism $\varphi : R[r ^{-1}] \rightarrow \mathbb{F}_p$. Set $m=\lceil \alpha \log p\rceil$; then $m<\log p$ for all sufficiently large $p$. For $\alpha$ small enough, the height estimate \ref{item:ht.X} implies $C^{\h_S(g)}<p$, for every $g\in X^{2m}$, so Theorem \ref{thm:ht.exists}\eqref{item:ht.mod.p} and Lemma \ref{lem:random.specialization.nonzero} imply that
\[
\Pr \left( \text{$\varphi$ is not injective on $X^m$} \right) < |X|^{2m} \frac{C m}{p}.
\]
Decreasing $\alpha$ if needed, we have
\begin{equation} \label{eq:phi.1-1}
\Pr \left( \text{$\varphi$ is injective on $X^m$} \right)>1-\frac{1}{2p^{1- \eta}}.
\end{equation}
$\mathbb{V}_{\mathbb{C}}:=\mathbb{V}\times_{\Spec R}\Spec\mathbb{C}$ is, by Lemma \ref{lem:V}, a proper closed subvariety of $G\times G$. Let $C_{EMO}$ be the constant obtained by applying Lemma \ref{lem:EMO} to the algebraic group $G\times G$ and the subvariety $\mathbb{V}_{\mathbb{C}}$. Similarly to \eqref{eq:phi.1-1}, if $\alpha$ is small enough then \ref{item:ht.X}, \ref{item:ht.V}, and Lemma \ref{lem:random.specialization.nonzero} imply that
\begin{equation} \label{eq:phi.V}
\Pr \left( (\forall g,h\in X^{mC_{EMO}}) \text{ if $(g,h)\notin \mathbb{V}(\mathbb{C})$ then $(\varphi(g),\varphi(h))\notin \mathbb{V} ^ \varphi(\mathbb{F}_p)$} \right)>1-\frac{1}{2p^{1- \eta}}.
\end{equation}
Thus, in order to prove Claim \ref{item:prob.SSA.Super} of Theorem \ref{thm:prob.SSA}, it is enough to show that if $\varphi$ satisfies both properties \eqref{eq:phi.1-1} and \eqref{eq:phi.V}, then it satisfies property \eqref{eq:RW.Delta}.

Let $\Delta \subsetneq \mathbb{G} ^ \varphi(\mathbb{F}_p)$ and denote $Y:=X^m \cap \varphi ^{-1} (\Delta)$. We claim that $\langle Y \rangle$ is not Zariski dense. Assuming the contrary, the set $Y \times Y$ generates a Zariski dense subgroup of $G \times G$, so Lemma \ref{lem:EMO} implies that there is a pair $(g,h)\in Y^{C_{EMO}} \times Y^{C_{EMO}}$ such that $(g,h)\notin \mathbb{V}(\mathbb{C})$. Since $\varphi$ satisfies property \eqref{eq:phi.V}, we get that $(\varphi(g),\varphi(h))\notin \mathbb{V}^{\varphi}(\mathbb{F}_p)$, so $\varphi(g),\varphi(h)$ generate $\mathbb{G} ^ \varphi(\mathbb{F}_p)$, a contradiction.

Let $H=\overline{\langle Y\rangle}^Z$. As above, $H\times H\subseteq\mathbb{V}_\mathbb{C}$. Applying Lemma~\ref{lem:quant-nonconcentration} to $G \times G$, $X \times X$, and $\mathbb{V}_\mathbb{C}$, there are constants $C_1,c_1$ such that
\[
\mu_X^{*m}(H(\mathbb{C}))^2=\mu_{X \times X}^{*m}((H \times H)(\mathbb{C})) \leq C_1^2e^{2c_1m},
\]
so $\mu_X^{*m}(H(\mathbb{C})) \leq C_1e^{c_1m}$. It follows that
\begin{align*}
\mu_{\varphi(X)}^{*m}(\Delta)
&=\mu_X^{*m}(\varphi^{-1}(\Delta)\cap X^m)\\
&\leq\mu_X^{*m}(H(\mathbb{C}))\\
&\leq C_1e^{-c_1m}\\
&\leq C_1p^{-\alpha c_1},
\end{align*}
where the equality uses the injectivity of $\varphi$ on $X^m$. For all sufficiently
large $p$, the last expression is at most $p^{-\alpha c_1/2}$. Thus
\eqref{eq:RW.Delta} holds with $\delta=\alpha c_1/2$.


\section{Mixed identities of linear groups}
\subsection{Necessary conditions.}

\begin{lemma}\label{lem:4.5}
Let $\gC$ be a group satisfying one of the following properties:
\begin{enumerate}
\item $\gC$ has a nontrivial virtually solvable subgroup $A$ whose normalizer $N_\gC(A)$ has finite index in $\gC$.
\item $\gC$ has a nontrivial subgroup $B$ whose normalizer $N_\gC(B)$ has finite index in $\Gamma$ such that the centralizer $C_\gC(B)$ is nontrivial.
\end{enumerate}
Then $\gC$ satisfies a nontrivial mixed identity.
\end{lemma}

\begin{proof}
$(1)$ Let $k=|\gC: N_\gC(A)|!$. If $A$ is finite then, for every $a\in A \smallsetminus 1$, $\gC$ satisfies the mixed identity $[x^k,a]^{|A|}=1$. Suppose now that $A$ is infinite. Being virtually solvable, $A$ has a nontrivial abelian characteristic subgroup $A_0$. For every $a\in A_0 \smallsetminus 1$, $\gC$ satisfies the mixed identity $[x^kax^{-k},a]=1$.

$(2)$ Let $k=|\gC:N_\gC(B)|!$. For every $b\in B\setminus\{1\}$ and $c\in C_\gC(B)\setminus\{1\}$, $\gC$ satisfies the mixed identity $[x^kbx^{-k},c]=1$.
\end{proof}

Recall that by the Tits alternative an amenable linear group is virtually solvable. Lemma \ref{lem:4.5} implies that a MIF linear group has a trivial amenable radical.

Let $\gC\subset\GL_d(\BC)$ be a MIF group, let $G$ be the Zariski closure of $\gC$, and let $G^\circ$ be the identity connected component. Note that $\gC$ intersects trivially the center of $G$. By Item (2) of Lemma \ref{lem:4.5}, $\gC$ intersects nontrivially at most one of the simple factors of $G^\circ$. If there is such a factor $G^\circ_0$ then $G^\circ_0$ as well as the product $\prod_{i\ne 0}G^\circ_i$ of all the other factors are normal in $G$ hence, by Item (2) of \ref{lem:4.5} again, $\gC$ projects faithfully modulo that product. 

In particular, we deduce:

\begin{prop}\label{prop:3.1}
Let $\gC$ be a MIF linear group. Then there is a faithful representation $\gC\hookrightarrow \GL_d(\BC)$ such that the identity component of the Zariski closure $\overline{\gC}^Z$ is center free and semisimple.
\end{prop}

Using Lemma \ref{lem:4.5}, one can follow, word-by-word, the argument in \cite[\S 7.1]{primitive} to deduce that a MIF linear group satisfies the linear condition to primitivity \cite[Definition 1.8 \& Theorem 1.9]{primitive}. In particular, 
\begin{cor}
A finitely generated MIF linear group is primitive.
\end{cor}

\subsection{Sharply MIF implies linearly MIF}

In this subsection, the groups are not assumed to be linear.

\begin{prop}\label{prop:Sharply->linear}
Let $\Gamma=\langle X\rangle,~|X|<\infty$ be a sharply MIF group. Then $\gC$ is linearly MIF.
\end{prop}

\begin{lemma}
Let $\gC$ be a group with more than two elements.
Let $\gc_1,\gc_2$ be two distinct nontrivial elements in $\gC$. 
Then for any $w_1,w_2\in\gC*\langle x\rangle\smallsetminus\{1\}$ there are $\ga_1,\ga_2\in \{1,\gc_1,\gc_2,x^{\pm 1}\}$ such that $[w_1^{\ga_1},w_2^{\ga_2}]\ne 1$.
\end{lemma}

\begin{proof}
It is not hard to verify that one can chose $\ga_i$ so that the reduced form of $w_1^{\ga_1}$, in the free product $\gC*\langle x\rangle$,
 starts and ends with a non-trivial power of $x$, while the reduced form of $w_2^{\ga_2}$ starts and ends with non-trivial elements of $\gC$. Such elements do not commute in $\gC *\langle x\rangle$.
\end{proof}

\begin{proof}[Proof of Proposition \ref{prop:Sharply->linear}]
Replacing $X$ by another finite generating set does not change the growth type of the functions under consideration. We may therefore assume that $X$ is symmetric and contains two distinct nontrivial elements $\sigma_1,\sigma_2$. It is enough to prove the required estimate for all sufficiently large $n$, since the finitely many remaining values can be absorbed into the final constant.

Fix $n$ and enumerate the nontrivial elements
\[
w_1,\ldots,w_k\in\Gamma*\langle x\rangle
\]
of length at most $n$ with respect to $X\cup\{x,x^{-1}\}$. If $q=|X|+2$, then $k\leq q^{n+1}$. Pair the $w_i$ consecutively and apply the preceding lemma to each pair, using conjugators from
\[
\{1,x,x^{-1},\sigma_1,\sigma_2\},
\]
to obtain nontrivial commutators. If the number of words is odd, carry the final word unchanged to the next stage. Repeat the construction. After at most
\[
m=\lceil\log_2 k\rceil
\]
stages we obtain one nontrivial word $W$ involving every $w_i$.

We record the length estimate. Let $L_j$ be an upper bound for the lengths of the words after $j$ stages, so $L_0=n$. Conjugation by one of the five elements above adds at most $2$ to the length, and therefore
\[
L_{j+1}\leq4L_j+8.
\]
It follows that
\[
|W|\leq4^m(n+3)\leq4k^2(n+3).
\]
Consequently,
\[
\log|W|\leq\log(4(n+3))+2(n+1)\log q\leq C_Xn
\]
for all sufficiently large $n$, where $C_X$ depends only on $X$.

Since $\Gamma$ is sharply MIF, there is $\gamma\in\Gamma$ with
\[
\|\gamma\|_X\leq c_X\log|W|\leq c_XC_Xn
\]
and $W(\gamma)\neq1$. By construction, if $w_i(\gamma)=1$ for any $i$, then the commutator containing $w_i$ at the first stage evaluates to $1$, and this identity propagates through every subsequent stage, forcing $W(\gamma)=1$. Hence $w_i(\gamma)\neq1$ for all $i$, proving that $\Gamma$ is linearly MIF.
\end{proof}

\begin{cor}\label{Cor:simultaneously}
Let $\Gamma=\langle X\rangle,~|X|<\infty$ be an MIF group. Then for any finite set of non-trivial elements $w_1,\ldots,w_k\in\gC*\langle x\rangle$, there is $\gc\in\gC$ with $w_i(\gc)\ne 1,~\forall i=1,\ldots k$.
\end{cor}

\subsection{Proof of Theorem \ref{thm:main}}
By Proposition \ref{prop:Sharply->linear}, it is enough to prove the first claim of the theorem.

Put $G:=\overline{\Gamma}^{Z}$. By Proposition \ref{prop:3.1}, after replacing the given faithful representation by another one, we may assume that the identity component $G^{\circ}$ is semisimple. Let
\[
\pi:\widetilde G\longrightarrow G^{\circ}
\]
be its simply connected covering. We recall that every complex semisimple algebraic group admits a faithful finite-dimensional representation; fix one for $\widetilde G$. Put
\[
\Gamma^{\circ}:=\Gamma\cap G^{\circ},
\qquad
\widetilde\Gamma:=\pi^{-1}(\Gamma^{\circ}).
\]
Choose finite symmetric generating sets $X$ and $X^{\circ}$ for $\Gamma$ and $\Gamma^{\circ}$, respectively. Then
\[
\widetilde X:=\pi^{-1}(X^{\circ})
\]
is a finite symmetric generating set for $\widetilde\Gamma$. Let $(\delta_n)$ be the random walk on $\widetilde\Gamma$ whose increments are independent and uniformly distributed in $\widetilde X$.

For a mixed word $w$, write
\[
L(w):=\|w\|_{X\cup\{x,x^{-1}\}}.
\]
We will prove that there are constants $C_0,C'_0$, depending only on $X$, such that, for every nontrivial mixed word $w$, there is an integer
\[
n_w\leq C_0\log L(w)+C'_0
\]
for which
\begin{equation}\label{eq:pr.0.5}
\Pr\bigl(w(\pi(\delta_{n_w}))=1\bigr)<\frac12.
\end{equation}
It follows that there is an element $\widetilde{g} \in \widetilde{\Gamma}$ with 
\[
w(\pi(\widetilde{g}))\neq1 \qquad\text{and}\qquad \| \widetilde{g} \|_{\widetilde X} \leq C_0\log L(w)+C_0' \leq \left( C_0+\frac{C_0'}{\log 2} \right) \log L(w).
\]
The theorem now follows because $\|\pi(\widetilde{g})\|_X\leq \max\{\|h\|_X:h\in X^{\circ}\}\,\|\widetilde{g}\|_{\widetilde X}$.

Choose a finitely generated domain $R\subseteq\mathbb C$ over which the faithful embeddings of $G$ and $\widetilde G$, the morphism $\pi$, the sets $X$ and $\widetilde X$, and all the constants occurring in these data are defined. Applying Theorem \ref{thm:prob.SSA} to $R$ and $\widetilde X$, and then replacing $R$ by the indicated localization, we obtain a group scheme $\mathbb G$ over $R$, a finite Galois extension $F/\mathbb Q$, and constants $p_0,\epsilon>0$ such that, whenever $p>p_0$ splits in $F$ and $\varphi$ is uniformly distributed in $\Hom(R,\mathbb F_p)$,
\begin{equation}\label{eq:Pr.expander}
\Pr\left(\Cay\bigl(\mathbb G^{\varphi}(\mathbb F_p),\varphi(\widetilde X)\bigr)
\text{ is not an $\epsilon$-expander}\right)<p^{-1/2},
\end{equation}
\[
\varphi(\widetilde\Gamma)=\mathbb G^{\varphi}(\mathbb F_p),
\]
and the algebraic group $\mathbb G^{\varphi}$ is connected, simply connected, and semisimple. Denote $d_G:=\dim\widetilde G$. Lemma \ref{lem:size.G} gives
\begin{equation}\label{eq:size.GG}
(p-1)^{d_G}\leq|\mathbb G^{\varphi}(\mathbb F_p)|\leq(p+1)^{d_G}.
\end{equation}

Let $S=R[\mathbb G]$ be the coordinate ring of $\mathbb G$. After a further localization, $S$ is a domain, and Noether normalization allows us to assume that it is finite over a ring
\[
\mathbb Z[1/N][x_1,\ldots,x_t]
\]
with $N\geq2$. Let $F_S$ be the algebraic closure of $\mathbb Q$ in $\Frac(S)$, and choose a finite Galois extension $F'/\mathbb Q$ containing the compositum $F \cdot F_S$. Apply Theorem \ref{thm:ht.exists} to $S$ and $F'$, and let $\h_S$ be the resulting height. Enlarge a constant $C_S$ so that it works both in Theorem \ref{thm:ht.exists} and in Lemma \ref{lem:random.specialization.nonzero}, applied with $F'$.

After one more localization, the degrees of the projective closures of all fibers $\mathbb G^{\varphi}$ are bounded by a constant $B_G$. There are constants $C_1,C_2$, independent of $w$, with the following properties. Each entry of the regular map
\[
w\circ\pi:\mathbb G\longrightarrow\GL_d
\]
has degree at most $C_1L(w)$, and its height as an element of $S$ is at most $C_2L(w)$. The degree assertion follows by induction on the word length. The height assertion follows from the same induction using Theorem \ref{thm:ht.exists}\eqref{item:ht.+}--\eqref{item:ht.*}.

Let $m_0=[\Gamma:\Gamma^{\circ}]$. Since $w(x^{m_0})$ is a nontrivial element of $\Gamma*\langle x\rangle$ and $\Gamma$ is MIF, there is $\gamma\in\Gamma$ such that $w(\gamma^{m_0})\neq1$. As $\gamma^{m_0}\in\Gamma^{\circ}$, the restriction of $w$ to $\Gamma^{\circ}$ is nontrivial. Consequently, $w\circ\pi$ is not identically the identity matrix on $\widetilde G$. Thus, for some $1\leq i,j\leq d$, the element
\begin{equation}\label{eq:def.s}
s:=(w\circ\pi)_{i,j}-\delta_{ij}\in S
\end{equation}
is nonzero. Notice that
\[
\deg(s)\leq C_1L(w),
\qquad
\h_S(s)\leq C_2L(w)+O(1).
\]

Choose once and for all a constant
\[
D>\max\{10C_SC_2,\,8C_1B_G,\,N\}.
\]
For all sufficiently large $L(w)$, Theorem \ref{thm:ht.exists}\eqref{item:ht.mod.p.2}, applied with $n=DL(w)$, provides a prime $p$ satisfying
\begin{equation}\label{eq:p.range}
DL(w)\leq p\leq C_SDL(w),
\end{equation}
such that $p$ splits completely in $F'$ and the image of $s$ in $S\otimes\mathbb F_p$ is not a zero divisor. If $\psi$ is uniformly distributed in $\Hom(S,\mathbb F_p)$, Lemma \ref{lem:random.specialization.nonzero} therefore yields
\begin{equation}\label{eq:psi.s.0}
\Pr(\psi(s)=0)<\frac{C_SC_2L(w)}p.
\end{equation}

For $\varphi\in\Hom(R,\mathbb F_p)$, the extensions $\psi:S\to\mathbb F_p$ of $\varphi$ are naturally identified with the set $\mathbb G^{\varphi}(\mathbb F_p)$. If $s^{\varphi}$ vanishes identically on $\mathbb G^{\varphi}$, then every such extension satisfies $\psi(s)=0$. Comparing the cardinalities of the fibers of the restriction map $\Hom(S,\mathbb F_p)\to\Hom(R,\mathbb F_p)$ by \eqref{eq:size.GG}, and using \eqref{eq:psi.s.0}, gives, for all sufficiently large $p$,
\begin{equation}\label{eq:phi.s.vanish}
\begin{split}
\Pr\bigl(s^{\varphi}\text{ vanishes identically on }\mathbb G^{\varphi}\bigr)
&<\left(\frac{p+1}{p-1}\right)^{d_G}\frac{C_SC_2L(w)}p\\
&<\frac13.
\end{split}
\end{equation}

Since $p^{-1/2}<2/3$ for all sufficiently large $p$, equations \eqref{eq:Pr.expander} and \eqref{eq:phi.s.vanish} imply that there is a homomorphism $\varphi:R\to\mathbb F_p$ for which the Cayley graph
\[
\Cay\bigl(\mathbb G^{\varphi}(\mathbb F_p),\varphi(\widetilde X)\bigr)
\]
is an $\epsilon$-expander and $s^{\varphi}$ does not vanish identically. The zero set of $s^{\varphi}$ in $\mathbb G^{\varphi}$ has degree at most $C_1B_GL(w)$ by B\'ezout's theorem. Lemma \ref{lem:DKL} and \eqref{eq:size.GG} therefore give, after increasing $p_0$ if necessary,
\begin{equation}\label{eq:zero-density}
\frac{|\{g\in\mathbb G^{\varphi}(\mathbb F_p):s^{\varphi}(g)=0\}|}
{|\mathbb G^{\varphi}(\mathbb F_p)|}
\leq\frac{C_1B_GL(w)}p.
\end{equation}

Set
\[
n_w:=\left\lceil-\frac{d_G\log(p+1)}{\log(1-\epsilon)}\right\rceil.
\]
By \eqref{eq:size.GG}, this is at least the mixing time required in Lemma \ref{lem:RW.general}. Since $w(\pi(\delta_{n_w}))=1$ implies $s(\delta_{n_w})=0$, that lemma and \eqref{eq:zero-density} show that
\[
\begin{split}
\Pr\bigl(w(\pi(\delta_{n_w}))=1\bigr)
&\leq
2\frac{|\{g\in\mathbb G^{\varphi}(\mathbb F_p):s^{\varphi}(g)=0\}|}
{|\mathbb G^{\varphi}(\mathbb F_p)|}\\
&\leq\frac{2C_1B_GL(w)}p<\frac12.
\end{split}
\]
Finally, the upper bound in \eqref{eq:p.range} gives
\[
n_w\leq
\frac{d_G}{-\log(1-\epsilon)}\log L(w)
+\frac{d_G\log(2C_SD)}{-\log(1-\epsilon)}+1.
\]
This proves \eqref{eq:pr.0.5} and hence the theorem.

\end{document}